\documentclass[11pt, reqno]{amsart}
\usepackage{amscd} 
\usepackage{amsfonts} 
\usepackage{amssymb} 
\usepackage{latexsym} 
\usepackage[arrow, matrix, curve]{xy}

\input diagrams

\newcommand{\ncm}{\newcommand}

\newtheorem{theorem}{Theorem}[section]
\newtheorem{prop}[theorem]{Proposition}
\newtheorem{lemma}[theorem]{Lemma}
\newtheorem{cor}[theorem]{Corollary}
\newtheorem{lem&def}[theorem]{Lemma \& Definition}
\newtheorem{definition}[theorem]{Definition}
\newtheorem{example}[theorem]{Example}

\def\M{\mathcal{M}}

\def\C{\mathbb{C}\,} 
\def\Z{\mathbb{Z}\,} 
 
\def\N{\mathbb{N}\,}

\ncm{\End}{\mbox{\rm End}\,}

\def\Hom{\mbox{\rm Hom}}

\def\|{\, | \,}

\def\id{\mbox{\rm id}}

\def\into{\hookrightarrow}

\ncm{\rarr}[1]{\stackrel{#1}{\longrightarrow}}
\ncm{\larr}[1]{\stackrel{#1}{\longleftarrow}}

\def\-2{_{(-2)}}
\def\-1{_{(-1)}}
\def\0{_{(0)}}
\def\1{_{(1)}}
\def\2{_{(2)}}
\def\3{_{(3)}}

\def\du1{\hat 1}

\begin{document}
\title[Uniquely separable extensions]{Uniquely separable extensions}
\author[L.~Kadison]{Lars Kadison} 
\address{Department of Mathematics \\ 209 S.\ 33rd St.\\ David Rittenhouse Lab \\
Philadelphia, PA 19104} 
\email{lkadison@math.upenn.edu} 
\thanks{}
\date{} 

\begin{abstract}
The separability tensor element of a separable extension of noncommutative rings is an idempotent when viewed in the correct endomorphism ring; so one speaks of a separability idempotent, as one usually does for separable algebras. 
It is proven that this  idempotent is full if and only the H-depth is 1 (H-separable extension).  Similarly, a split extension has 
a bimodule projection; this idempotent is full if and only if the ring extension has depth 1 (centrally projective extension). Separable and split extensions have separability idempotents and bimodule projections in 1 - 1 correspondence via an endomorphism ring theorem in Section~3.  If the separable idempotent is unique, then the separable extension is called uniquely separable. A Frobenius extension with  invertible $E$-index is uniquely separable if the centralizer equals the center of the over-ring.  It is also shown that
a uniquely separable extension of semisimple complex algebras with invertible E-index has depth  1.  
 Earlier group-theoretic results are recovered and related to depth $1$. The dual notion, uniquely split extension, only occurs trivially for finite group algebra extensions over complex numbers. 
\end{abstract} 
\maketitle

\section{Introduction and Preliminaries}
\label{one}

The classical notion of separable algebra is one of a semisimple algebra that remains semisimple under every base field extension.  The  approach of Hochschild to Wedderburn's theory of associative algebras in the Annals was cohomological, and characterized a separable $k$-algebra $A$ by having a separability idempotent in $A^e = A \otimes_k A^{\rm op}$.  A separable extension of noncommutative rings is characterized similarly by possessing separability elements in \cite{HS}:  separable extensions are shown to be left and right semisimple extensions in terms of Hochschild's relative homological algebra (1956).  In related developments around 1960, also separable
algebras over commutative rings, Galois and Brauer theory of commutative rings were first defined and studied by Auslander and Goldman.   It is pointed
out in Section 1 of this paper that the separability element of a (unital) ring extension of noncommutative rings
$B \into A$ is  an idempotent in $(A \otimes_B A)^B$ ($ \cong \End {}_AA \otimes_B A_A$).  We prove that this idempotent is full if and only if the extension is H-separable, a strong condition generalizing the notion of Azumaya algebra.      

Separability idempotents in $A \otimes_B A$ for a separable extension $A \supseteq B$ are generally not 
unique.  For example, if $A = M_n(B)$ there are $n$ different separability idempotents, and their convex combinations, defined as in Example~\ref{example-matrixalgebra}, so that uniqueness can only happen for $n = 1$.  For example, given a finite-dimensional  group algebra $A = kG$ over an algebraically closed field $k$ with characteristic not dividing the order of $G$, one sees from the Wedderburn decomposition of semisimple algebras that $A$ has unique separability idempotent if and only if $G$ is abelian. The papers \cite{SH,S} describe a condition on a subgroup $H$ of a finite
group $G$  that is equivalent to $A = KG \supseteq B = KH$ having a unique separability element
where $K$ is a commutative ring in which $|G:H|1$ is invertible:  their condition  is that each conjugacy class of $G$ be an $H$-orbit. For example, this happens if $G = HZ(G)$ where $Z(G)$ is the center of $G$.  In Section~\ref{three} we generalize this to a separable Frobenius extension $A \supseteq B$ having invertible $E$-index in the center $Z(A)$.  We show in Theorem~\ref{th-SH} that $A$ has a unique separability element over $B$ if and only if the centralizer
$A^B = Z(A)$.  If $K$ is an algebraically closed field of characteristic zero, it follows from the trivial observation 
$Z(B) \subseteq A^B = Z(A)$ and Burciu's characterization of depth one in \cite{B} that $A$ is centrally projective over $B$:  ${}_BA_B \oplus * \cong n \cdot {}_BB_B$, or $A$ as a natural $B$-bimodule is isomorphic to a direct summand of a finite direct sum of copies of $B$.   It is shown in \cite{LK2012} that such extensions automatically
satisfy ${}_BB_B \oplus * \cong {}_BA_B$, so that   a centrally projective extension $A \supseteq B$ is characterized by the bimodules ${}_BA_B$ and ${}_BB_B$ being similar \cite{AF}.  

Centrally projective extensions are the depth $1$ case of odd minimal depth $d(B,A) = 2n+1$ where a 
ring extension $A \supseteq B$ has similar (natural) $B$-$B$-bimodules $A^{\otimes_B n}$
and $A^{\otimes_B (n+1)} := A \otimes_B \cdots \otimes_B A$ ($n+1$ times $A$).  The H-separable extension $A \supseteq B$ defined in \cite{Hir} is the H-depth $1$ case of odd minimal H-depth $d_H(B,A) = 2n-1$ where $A^{\otimes_B n} \sim A^{\otimes_B (n+1)}$ as $A$-bimodules.  See below
in this section for more details on depth and H-depth.  As mentioned, we show in two propositions of Section 1 that H-depth one and depth one ring extensions are characterizable in terms of a full idempotent in an endomorphism ring above the ring extension (in fact, isomorphic to the centralizer rings of the Jones tower that receive Hopf structure in  \cite{KS}).  
In Sections~4 and~5 of this paper we show that unique separable subalgebra pairs of semisimple complex algebras with invertible $E$-index have depth $1$.  This is compared to results of \cite[Singh, Hanna]{SH,S} and \cite[Boltje-K\"ulshammer]{BK}
on uniquely separable group ring extensions and  subgroups of finite groups having depth $1$. We show in Section~3 that unique separable extensions have endomorphism ring extensions that are uniquely split extensions, and we characterize uniquely split extensions of group algebras, or more generally split Frobenius extensions, by having centralizer equal to the center of the subalgebra.  We prove a triviality theorem for uniquely split group algebra extensions over the complex numbers, again in Section~3. 

\subsection{Separable extensions and finite-dimensional algebra extenions}
A ring extenions $R \supseteq S$ is a separable extension if $\mu: R\otimes_S R \rightarrow R$
splits as an $R$-$R$-bimodule epimorphism.  This is clearly equivalent to there being an element
$e \in R \otimes_S R$ such that $re = er$ for every $r \in R$ and $\mu(e) = 1$.
(Briefly, one writes $e \in (R \otimes_S R)^R$ and $e^1e^2 = 1$ and calls such an $e$ a separability element.)  Separable extensions are characterized by having relative Hochschild cohomological dimension zero, and  a separable extension $R \supseteq S$ satisfies the inequality in right global dimension of rings $D(S) \geq D(R)$ if $R_S$
is projective \cite{HS}. 

Note that the $S$-central elements in $R \otimes_S R$, denoted by $T := (R \otimes_S R)^S$ are isomorphic to the endomorphism ring $\End {}_RR \otimes_S R_R$
via $$t  \longmapsto (r \otimes_S r' \mapsto rtr').$$
The inverse mapping is of course $F \mapsto F(1 \otimes 1)$.  This transfers a ring
structure onto $T$ given in Sweedler-like notation by
$t', t = t^1 \otimes t^2 \in (R \otimes_S R)^S$ and
$$ tt' = {t'}^1 t^1 \otimes   t^2 {t'}^2, \ \ \ 1_T = 1_R \otimes_S 1_R.$$
A separability element $e = e^1 \otimes_S e^2 \in (R \otimes_S R)^R \subseteq T$ and satisfies $e^2 = e$ in this multiplication, since $e^1 e^2 = 1_R$.  Thus it
makes sense to continue calling it a \textit{separability idempotent}, in continuation of the terminology in the special case of separable algebras (over a commutative ring)  when $S \subseteq Z(R)$, the center of $R$.  

The rest of the subsection studies  theorems that are useful for identifying when subalgebras of a finite-dimensional algebra form a separable extension, or not.  When the subalgebra is trivially one-dimensional, the important theorem for identifying a separable algebra states the following \cite[10.7]{P}:
"An algebra is separable if and only if it is semisimple with block matrix algebras over division algebras whose centers are separable field extensions of the ground field ."  The following lemma compiles the most useful results towards this end from \cite{HS} and \cite{CK}, with easy proofs (for not necessarily finite-dimensional algebras).

\begin{lemma}
The following holds for a separable extension $A \supseteq B$.  
\begin{itemize}
\item (A) If $\pi: A \rightarrow A'$ is an algebra epimorphism, with $\pi(B) = B'$, then $A'$ is a separable extension of $B'$.
\item (B) If $T$ is an intermediate ring $A \supseteq T \supseteq B$, then $A$ is a separable extension of $T$.
\item (C) If there is an algebra epimorphism $\pi: A \rightarrow B$, which splits $B \into A$, and $I$ denotes $\ker \pi$, then $I^2 = I$.  
\item (D) Suppose $I$ is an ideal of $A$,  and subalgebras $A' \supseteq B'$ satisfy $A = A' \oplus I$
and $B = B' \oplus I$. Then $A' \supseteq B'$ is a separable extension if and only if $A \supseteq B$
is a separable extension.
\end{itemize}
\end{lemma}

\begin{prop}
\label{prop-rad}
Suppose $A$ is a finite-dimensional algebra over a perfect field with subalgebra $B$ with radical
$J(B)$ an ideal of $A$.  Then $A \supseteq B$ is a separable extension if and only if $J(A) = J(B)$.
\end{prop}
\begin{proof}
Suppose $J(A) = J(B) = J$.  Then $J^n = 0$ for some $n \in \N$.  By the Wedderburn principal theorem \cite[11.6]{P}, there is a separable subalgebra $A'$ such that $A = A' \oplus J$.  Let $B' = B \cap A'$,
a subalgebra of $A'$ forming a separable extension by Lemma (B).  Since $B = B' \oplus J$, it follows
from Lemma (D) that $A$ is a separable extension of $B$.  

Suppose $A$ is separable extension of $B$ where $J(B)$ is an ideal in $A$.  Since $J(B)$ is nilpotent,
one has $J(B) \subseteq J(A)$.  Let $A' = A/J(B)$ and $B' = B/J(B)$.  By Lemma (A), $A'$ is a separable extension of $B'$.  The radical of $A'$ is $I := J(A)/J(B)$, and is the kernel of the canonical epi $A \rightarrow A'$.  Since $B'$ is a separable algebra (over a perfect field), the canonical epi splits as an algebra mapping, so that Lemma (C) implies $I^2 = I$.  Suppose $J(A)^n = 0$.  Then $I = I^{2^n} = 0$, whence $J(A) = J(B)$.  
\end{proof}

The proposition also works without the hypothesis on the ground field if $A/J(A)$ and $B /J(B)$ are known to be separable algebras.  This follows from Hochschild's proof extending the Wedderburn principal theorem \cite{P}.  
\begin{example}
\begin{rm}
Let $A = T_n(k)$ be the algebra upper-triangular $n$-by-$n$ matrices over any field $k$, $B = U(n)$ be
the subalgebra strictly upper-triangular matrices with $1$.  It follows from $J(A) = J(B)$ and the proposition that $A$
is a separable extension of $B$. 

Suppose  $C = k1 + k(e_{12}+ \cdots + e_{n-1,n}) + \cdots + ke_{1n}$ be the Jordan subalgebra $J_n(k) $ within $B$ and $A$.  (Note that $D(A) = 1$ and $D(C) = \infty$, since $A$ is hereditary and $C \cong k[X]/(X^n)$ is Frobenius.)  Except in the case $n = 1,2$, it is a thornier problem to determine  separability of the extension $A \supseteq C$, since 
$A \supseteq C$ may be either separable or inseparable in a general tower $A \supseteq B \supseteq C$, when $A \supseteq B$ is separable
and $B \supseteq C$ is inseparable (i.e., not a separable extension). 
\end{rm}
\end{example}
Even when the radical of $B$ is not an $A$-ideal, it is often enough to maneuver as follows with the information in the Lemma above.  
\begin{example}
\begin{rm}
Let $k$ be a field and $B$ be the Sweedler-Nakayama algebra $B = k(e_{11} + e_{44}) + k(e_{22} + e_{33}) + ke_{21} + ke_{43}$ in terms of matrix units in $M_4(k)$.  Consider $B$ as a subalgebra
of $A = ke_{11} + ke_{22} + ke_{33} + ke_{44} + ke_{31} + ke_{41} + ke_{42} + ke_{43}$,
a structural matrix subalgebra of $M_4(k)$.  Note that $J(B)$ does not satisfy the hypothesis of the
proposition, so we augment it to $J = ke_{21} + ke_{43} + ke_{41}$, an $A$-ideal such that
$J \cap B = J(B)$.  Then $A \rightarrow A/J$ maps $B$ onto $B' \cong k^2$, a separable $k$-algebra,
but $A' := A/J$ has radical $I = ke_{31} + ke_{42}$.  If $A'$ is separable over $B'$, then it is separable over its diagonal subalgebra $\cong k^4$; by Lemma (D) this implies $I = 0$ a contradiction.  Then
$A'$ is not separable over $B'$ and by Lemma (A), $A$ is not separable over $B$.
\end{rm}
\end{example}

\begin{example}
\label{example-matrixalgebra}
\begin{rm}
Let $K$ be a commutative ring and $A = M_n(K)$ the full $K$-algebra of $n \times n$ matrices.  Let $e_{ij}$ denote the matrix units ($i,j = 1,\ldots, n$). Any of the
$n$ elements $e_j = \sum_{i=1}^n e_{ij} \otimes_K e_{ji}$ are separability idempotents for $A$ (as well as any convex combination of  $e_1,\ldots,e_n$). 
\end{rm}
\end{example}

\subsection{Preliminaries on subalgebra depth} Let $A$ be a unital associative ring.  The category of  right modules over $A$ will be denoted
by $\M_A$.  
Two modules $M_A$ and $N_A$ are \textit{H-equivalent} (or similar) if $M \oplus * \cong  N^q $ and $N \oplus * \cong  M^r$
for some $r,q \in \N$ (sometimes briefly denoted by  $M \sim N$). It is well-known that H-equivalent modules have Morita equivalent endomorphism rings.  

 Let $B$ be a subring of $A$ (always supposing $1_B = 1_A$).  Consider the natural bimodules ${}_AA_A$, ${}_BA_A$, ${}_AA_B$ and
${}_BA_B$ where the last is a restriction of the preceding, and so forth.  Denote the tensor powers
of ${}_BA_B$ by $A^{\otimes_B n} = A \otimes_B \cdots \otimes_B A$ for $n = 1,2,\ldots$, which is also a natural bimodule over  $B$ and $A$ in any one of four ways;     set $A^{\otimes_B 0} = B$ which is only a natural $B$-$B$-bimodule.  

\begin{definition}
\label{def-depth}
If $A^{\otimes_B (n+1)}$ is H-equivalent to $A^{\otimes_B n}$ as $X$-$Y$-bimodules,  one says $B \subseteq A$ has  
\begin{itemize}
\item depth $2n+1$ if $X = B = Y$;
\item left depth $2n$ if $X = B$ and $Y = A$;
\item right depth $2n$ if $X = A$ and $Y = B$;
\item H-depth $2n-1$ if $X = A = Y$.
\end{itemize}
valid for even depth and H-depth if $n \geq 1$ and for odd depth if $n \geq 0$. 
\end{definition}

 For example, $B \subseteq A$ has depth $1$ iff ${}_BA_B$ and ${}_BB_B$ are H-equivalent \cite{BK2}.
Equivalently, 
\begin{equation}
\label{eq: d1}
{}_BA_B \oplus * \cong n \cdot {}_B{B}_B
\end{equation} for some $n \in \N$ \cite{LK2012}.  This in turn is
equivalent to there being $f_i \in \Hom ({}_BA_B, {}_BB_B)$ and $r_i \in A^B$ such that
$\id_A = \sum_i f_i(-) r_i$, the classical \textit{central projectivity} condition \cite{Mu}.  
  In this case, it is easy to show that $A$ is ring isomorphic to $B \otimes_{Z(B)} A^B$ where
$Z(B), A^B$ denote the center of $B$ and centralizer of $B$ in $A$. From this one deduces  
\begin{lemma}
\label{lemma-d1}
A centrally projective ring extension $A \supseteq B$ (or depth $1$ extension) has centers
satisfying  $Z(B) \subseteq Z(A)$.  
\end{lemma}

Another example, $B \subset A$ has right depth $2$ iff ${}_AA_B$
and ${}_A A \otimes_B A_B$ are similar.  If $A = \C G$ is a group algebra of a  finite group $G$ and $B = \C H$ is a group algebra of a subgroup $H$ of $G$, then $B \subseteq A$ has right depth $2$ iff $H$ is a normal subgroup of $G$ iff $B \subseteq A$ has left depth $2$ \cite{KK}; a similar statement of normality is true for a Hopf subalgebra $R \subseteq H$ of finite index and over any field \cite{BK}. Depth two is the key condition that generates Hopf algebroid structures on certain endomorphism rings derived from the ring extension $A \supseteq B$ that occur
in the next section.   

Note that $A^{\otimes_B n}\oplus * \cong A^{\otimes_B (n+1)}$ for all $n \geq 2$
and in any of the four natural bimodule structures: one applies $1_A$ and multiplication to obtain a split monic, or split epi oppositely. For three of the bimodule structures, it is true for $n =1$;  as $A$-$A$-bimodules, equivalently $A \oplus * \cong A \otimes_B A$ as $A^e$-modules, this is the separable extension condition on $B \subseteq A$.  
Ring extension theorists will recognize \begin{equation}
\label{eq: H-sep}
{}_AA \otimes_B A_A \oplus * \cong q \cdot {}_AA_A
\end{equation}
 for some $q \in \N$
as the H-separability condition, which implies $A$ is a separable extension of $B$ \cite{NEFE}.  Somewhat similarly, ${}_BA_B \| q  \cdot {}_BB_B$ implies
${}_BB_B \| {}_BA_B$ \cite{LK2012}. It follows that subalgebra depth and H-depth may be equivalently defined by replacing the similarity bimodule conditions for depth and H-depth in Definition~\ref{def-depth} with the corresponding bimodules on 
\begin{equation}
\label{eq: def}
A^{\otimes_B (n+1)} \oplus * \cong  q \cdot A^{\otimes_B n}
\end{equation}
 for some positive integer $q$ \cite{BDK, LK2011, LK2012}.  

Note that if $B \subseteq A$ has H-depth $2n-1$, the subalgebra has (left or right) depth $2n$ by restriction of modules.  Similarly, if $B \subseteq A$ has depth $2n$, it has depth $2n+1$.  If $B \subseteq A$ has depth $2n+1$, it has depth $2n+2$ by tensoring either $-\otimes_B A$ or $A \otimes_B -$ to $A^{\otimes_B (n+1)} \sim A^{\otimes_B n}$.     Similarly, if $B \subseteq A$ has left or right depth $2n$, it has H-depth $2n+1$.  Denote the minimum depth
of $B \subseteq A$ (if it exists) by $d(B,A)$ \cite{BDK}.  Denote the minimum H-depth of $B \subseteq A$ by $d_h(B,A)$.  Note that
$d(B,A) < \infty$ if and only if $d_H(B,A) < \infty$; in fact, $| d(B,A) - d_H(B,A)| \leq 2$ if either is finite.  

Returning briefly to the H-separable (or equivalently H-depth 1) ring extension $B \subseteq A$ satisfying Eq.~(\ref{eq: H-sep}),  we note that $\Hom ({}_AA \otimes_B A_A, {}_AA_A) \cong A^B$ and
$\Hom ({}_AA_A, {}_AA \otimes_B A_A) \cong (A \otimes_B A)^A$, from which it follows that
an equivalent condition for H-depth 1 is that there are $2q$ elements $r_i \in A^B$ and $e_i \in (A \otimes_B A)^A$ such that 
$1 \otimes_B 1 = \sum_{i=1}^q r_i e_i$.  
\begin{lemma}
\label{lemma-hsep}
An H-depth 1 extension $B \subseteq A$ that is faithfully flat satisfies $Z(A) \subseteq Z(B)$.  
\end{lemma}
\begin{proof}
Suppose $x \in Z(A)$.  Then $$x \otimes_B 1 = \sum_i xr_i e_i = \sum_i r_i e_i x = 1 \otimes_B x$$
whence $x \in B$.  
\end{proof}

As examples of algebra extension depth, for the permutation groups $\Sigma_n < \Sigma_{n+1}$
and their corresponding group algebras $B \subseteq A$ over any
commutative ring $K$, one has depth $d(B,A) = 2n-1$ \cite{BKK,BDK}. Depths of subgroups in $PGL(2,q)$, twisted group algebras and Young subgroups of $\Sigma_n$ are computed in \cite{F, D, FKR}.    If $B$ and $A$ are semisimple complex algebras, the minimum odd depth is computed from powers of an order $r$ symmetric matrix with nonnegative entries $S := MM^t$ where $M$ is the inclusion matrix
$K_0(B) \rightarrow K_0(A)$ and $r$ is the number of irreducible representations of $B$ in a basic set of $K_0(B)$; the depth is $2n+1$ if $S^n$ and $S^{n+1}$ have an equal number of zero entries \cite{BKK}.
Similarly, the minimum H-depth of $B\subseteq A$ is computed from
powers of an order $s$ symmetric matrix $T = M^tM$, where $s$ is the rank of $K_0(A)$, and the power $n$ at
which the number of zero entries of $T^n$ stabilizes \cite{LK2012}. 
It follows that the subalgebra pair of semisimple complex algebras $B \subseteq A$   always has finite depth.


\section{Full idempotent characterizations of H-separable and centrally projective ring extensions}

In this section,  we give  characterizations of H-depth 1 and depth 1 ring extension in terms of well-known idempotents being full idempotents. Recall the equivalent form of Morita theory, which states that for a ring $A$ with idempotent
$e \in A$, $eAe$ is Morita equivalent to $A$ so long as $e$ is a \textit{full} idempotent, i.e. $AeA = A$: the Morita context bimodules are $eA$ and $Ae$ \cite{AF}.  Recall from the previous section  the ring structure on $T = (A \otimes_B A)^B$ of a ring extension $A \supseteq B$.  

\begin{prop}
\label{prop-sep}
A separable extension $R \supseteq S$ is H-separable if and only if
a separability idempotent $e \in T$ is full. In this case, the center $Z(R)$ is Morita equivalent to $\End {}_RR \otimes_S R_R$ \cite[Corollary 4.3]{KK}. 
\end{prop}
\begin{proof}
A ring extension $R \supseteq S$ is H-separable if and only if there are $q$ elements $e_i \in T$ and $q$ elements $c_i \in R^S$ (the centralizer of $S$ in $R$) such that
$1_T = \sum_{i=1}^q c_i e_i$ \cite{NEFE}, which may be seen in an exercise using Eq.~(\ref{eq: H-sep}) and using $q$ bimodule homomorphisms to and from ${}_RR_R$ and ${}_RR \otimes_S R_R$ composed and summing up to the identity
mapping on $R \otimes_S R$. It was  pointed out that H-separable extensions
are separable.  

A separability idempotent $e$ in $T$ is \textit{full} if $TeT = T$ (equivalently,
the left ideal $(R \otimes_S R)^R = Te$ is a progenerator). But the two-sided ideal 
$TeT = R^S (R \otimes_S R)^R$, since for any $t \in T$, $t^1 t^2 \in R^S$
(where $t = t^1 \otimes t^2$ suppresses a possible summation).  
Thus, $1_T \in TeT$ if and only if $R$ is an H-separable extension of $S$.  

Note that $eTe \cong Z(R)$ via $ef \mapsto f^1f^2$ via the multiplication mapping $\mu: R \otimes_S R \rightarrow R$, since $Te = (R \otimes_S R)^R$ and for $f \in (R \otimes_S R)^R$,  $ef = f^1f^2 e =
efe$. An inverse ring homomorphism is given by $z \mapsto ze$.  
\end{proof}
Recall that a ring extension $R \supseteq S$ is split if ${}_SS_S \oplus * \cong {}_SR_S$.  Equivalently,
there is a bimodule projection $E$ from $R$ onto $S$ satisfying $E^2 = E \in \End {}_SR_S := U$.  
\begin{prop}
\label{prop-split}
A split extension $R \supseteq S$ is centrally projective (or has depth one) if and only if there is a bimodule projection $E: R \rightarrow S$ that
is a full idempotent in $U$.  In this case, the center $Z(S)$ is Morita equivalent to $\End {}_SR_S$.  
\end{prop} 
\begin{proof}
Suppose $E$ is a full idempotent in $U$, i.e., $UEU = U$. Since $$UEU = \{ \sum_i \alpha_i \circ E \circ \beta_i \|  \alpha_i, \beta_i \in U \} = \{ \sum_i \alpha_i(1_R) E(\beta_i(-)) \|  \alpha_i, \beta_i  \in U \}$$
and $\alpha_i(1_R) \in R^S$, then $$\id_R = \sum_{i=1}^q \alpha_i \circ E \circ \beta_i = \sum_i f_i \circ g_i$$
where $f_i := \alpha_i \circ E \in \Hom ({}_SS_S, {}_SR_S)$ is in fact multiplication by $\alpha(1_R) $, and $g_i := E \circ \beta_i \in \Hom ({}_SR_S, {}_SS_S)$. Since $\id_R = \sum_i f_i \circ g_i$ is equivalent to Eq.~(\ref{eq: d1}), it follows that the ring extension has depth one. 

Conversely, suppose $R \supseteq S$ has depth one.  By \cite[Lemma 1.9]{LK2012} there is a bimodule projection $E: R \rightarrow S$, which we view as an idempotent in $U$. Moreover, via Eq.~(\ref{eq: d1}) there are $f_i \in \Hom ({}_SS_S, {}_SR_S), g_i \in \Hom ({}_SR_S, {}_SS_S)$ such that $\id_R = \sum_i f_i \circ g_i$.  Then
$$ \id_R = \sum_i (f_i \circ E) \circ E \circ (E \circ g_i) \in UEU. $$ 

The final statement follows from $EUE = \{ E(\alpha(1_R))E \|  \alpha \in U \}
= Z(S) E$, since $E(R^S) = Z(S)$, as well as  $Z(S) \cong Z(S)E$. 
\end{proof}

Recall that a ring $R$ with nontrivial idempotent $e$ has \textit{corner ring} $eRe$, to which it passes on several properties studied in \cite[Chapter 21]{Lam1} and \cite[16.25, 18.15]{Lam2}. This is  captured in the following definition, which is related to the weaker notion of "symmetric separably divides" in \cite[Def.\ 2.1]{LK2018}.

\begin{definition}
\begin{rm}
A ring $A$ \textbf{Morita divides} a ring $B$ if there is a bimodule homomorphism $\mu: {}_BQ \otimes_A P_B \rightarrow {}_BB_B$ and a bimodule isomorphism (or epimorphism) $\nu: {}_AP \otimes_B Q_A \rightarrow {}_AA_A$, which are associative with respect to application to $Q \otimes P \otimes Q \rightarrow Q$ and $P \otimes Q \otimes P \rightarrow P$. 
\end{rm}
\end{definition}
It follows  that 
\begin{enumerate}
\item in case $\nu$ is an epi, then it is an isomorphism; 
\item $P_B$ and ${}_BQ$ are finite projective modules;
\item ${}_AP$ and $Q_A$ are generators; 
\item the (natural) $A$-$B$-bimodule homomorphism $P \rightarrow \Hom ({}_BQ, {}_BB)$, given by $p \mapsto \mu(- \otimes_A p)$, and the $B$-$A$-homomorphism $Q \rightarrow \Hom (P_B, B_B)$ given by $q \mapsto \mu(q \otimes_A -)$, are both isomorphisms;
\item the natural ring homomorphisms $A \rightarrow \End P_B$ and $A \rightarrow \End {}_BQ$ are isomorphisms.  Compare \cite[Ex.\ 22.5]{AF}
\end{enumerate}

Clearly, $A$ and $B$ are Morita equivalent if $A$ Morita divides $B$ and $B$ Morita divides $A$ with the same context bimodules.  
 For example, a ring $R$ is Morita divided by a corner ring $eRe$ since the bimodule isomorphism $eR \otimes_R Re \stackrel{\cong}{\longrightarrow} eRe$, given by $er \otimes_R r'e \mapsto err'e$, is associative w.r.t. the bimodule homomorphism $\mu: Re \otimes_{eRe} eR \rightarrow R$ given by $re \otimes er' \mapsto rer'$: of course, $eR$ and $Re$ are projective ideals in $R$.  The corner ring $eRe$ Morita divides $R$ iff $e$ is a full idempotent, i.e., the mapping $\mu$ is epi. This example is well-representative of Morita divides, due to (5) above.  

 A special case of corner rings occurs for example in the finite-dimensional Hopf algebra $H$ action on algebra $A$, where there is an idempotent $e$ in the smash product $A \# H$ such that the subalgebra of invariants,
$A^H \cong e(A \# H)e$, in case the trace function $A \rightarrow A^H$ is surjective \cite[4.3.4, 4.5]{Mo}.

A closer look at the last paragraphs of the proofs of Propositions~\ref{prop-sep} and~\ref{prop-split} makes the following proposition obvious.

\begin{prop}
Suppose $A$ is a ring extension of $B$.  If $A$ is a separable extension of $B$, then $T$ Morita divides $Z(A)$.  If $A$ is a split extension of $B$, then $U$ Morita divides $Z(B)$.  
\end{prop}
\begin{proof}
A second proof is obtained from noting ${}_AA \otimes_B A_A \cong {}_AA_A \oplus *$ for a separable extension, then apply $\End$ and $\End {}_AA_A \cong Z(A)$, from which $Z(A)$ is a corner ring of $T$.  Similarly, if $A$ is a split extension of $B$, we have ${}_BA_B \cong {}_BB_B \oplus *$, then
$\End {}_BA_B$ has corner ring $Z(B)$.  
\end{proof}

\begin{example}
\begin{rm}
Consider the finite cyclic monoid $$\{ 1,f,f^2,\ldots,f^n,\ldots,f^{2n}\}$$ subject to the relation $f^n = f^{2n}$.  Choosing a field $k$, denote the commutative monoid algebra $M(2n,n)$ generated by $f$ of dim $2n$ with idempotent $e = f^n$. This is also a bialgebra \cite{Mo}.  Since $(ef)^n = ef^n = e$ and $[(1-e)f]^n = (1-e)e = 0$, the idempotent decomposition  shows that
\begin{equation}
\label{eq: interesting}
M(2n,n) = eM(2n,n) \oplus (1-e)M(2n,n) \cong k\Z_n \times k[X]/(X^n)
\end{equation}
(an interesting contrast to the Taft Hopf algebra $H_n \cong k[X]/(X^n) \# k\Z_n$ of dimension $n^2$
if $k$ contains a primitive root-of-unity).  
The subalgebra $B$ generated by $(1-e)f$ is isomorphic to $k[X]/(X^n)$, forms a separable extension with respect to $M(2n,n)$ by Proposition~\ref{prop-rad}, if $k$ has characteristic zero.  The algebra extension $M(2n,n) \supset B$ is also 
a split extension (as $B$-modules) with a nonprojective direct summand $eM(2n,n)$. Both $M(2n,n)$ and $B$ are symmetric algebras, so this algebra extension is  not Frobenius (an interesting supplement to the examples in \cite{CK}). In addition,
the bialgebra $M(2n,n)$ has corner algebra the (Hopf) group algebra $k\Z_n$.  
\end{rm}
\end{example}

As a final remark, the endomorphism rings of the ring extension, in which $e$ or $E$ are full idempotents, are the centralizers A and B, up to ring isomorphism and with suitable hypotheses on the ring extension, in the Hasse diagram of centralizers in \cite[Figure 1]{KN2}.  The centralizer algebras $A$ and $B$  receive two nondegenerate pairings that provide dual Hopf algebra  structures on these and show $R \supseteq S$ is a Hopf-Galois extension: the depth two condition required in all four publications  \cite{KS, KN, KN2, KN3} is automatically satisfied by the depth one or H-depth one conditions above.  The Hopf algebroid structures on $T$ and $U$ above are sketched in \cite[4.7, 4.8, 5.7, 5.8]{KS}.

\section{Endomorphism ring theorem for uniquely separable extension}

A \textit{uniquely separable extension} is a ring extension $A \| B$, which is a separable extension having a unique separability idempotent.  In this section we prove an endomorphism ring theorem for separable extensions $A \| B$ showing that separability idempotents are in one-to-one correspondence with bimodule projections from $\End A_B$ onto $A$ (embedded as  left multiplication endomorphisms).  For this reason, we  define
a \textit{uniquely split extension} $A \| B$ as a ring extension possessing a unique bimodule projection
(or conditional expectation) $\mathcal{E}: {}_BA_B \rightarrow {}_BB_B$ with respect to the natural bimodules and $E(1_A) = 1_B$ (where $1_B = 1_A$ perhaps through an identification such as $\lambda_{1_A} = \id_A$,
$\lambda$ left multiplication).   

\begin{theorem}
Suppose $A \| B$ is a ring extension such that the natural module $A_B$ is finitely generated projective.
Then separability idempotents of $A \| B$ are in 1-1 correspondence with bimodule projections
$\mathcal{E}: \End A_B \rightarrow A$.
\end{theorem}
\begin{proof}
Suppose that $\{x_i \}$ $\{ f_i \in \Hom (A_B, B_B) \}$ is a finite projective base for the module $A_B$. We then have
$a = \sum_i x_i f_i(a)$ for every $a \in A$ as well as $\alpha = \sum_i \alpha(x_i) f_i$ for every $\alpha \in \End A_B$.  

Given a separability idempotent $e = e^1 \otimes_B e^2$ in $(A \otimes_B A)^A$, define a mapping
$\mathcal{E}: \End A_B \rightarrow A$ by $\mathcal{E}(\alpha) = \alpha(e^1)e^2$ for every
$\alpha \in \End A_B$.  Since $e^1 e^2 = 1_A$, it follows that $\mathcal{E}(\id_A) = 1_A$. 
Moreover, for every $a, a' \in A$, observe that
$$\mathcal{E}(\lambda_a \alpha \lambda_{a'}) = a\alpha(a'e^1)e^2 = a\alpha(e^1)e^2a' = a\mathcal{E}(\alpha)a',$$
wherefore $\mathcal{E}$ is a bimodule projection.  This defines a mapping from separability idempotents to conditional expectations on the endomorphism ring.

Conversely, given a bimodule projection $\mathcal{F}: \End A_B \rightarrow A$ satisfying
$\mathcal{F}(\lambda_a) = a$ for every $a \in A$ and $\mathcal{F}(\lambda_a \alpha \lambda_{a'}) = a\mathcal{F}(\alpha)a'$ for
every $a,a' \in A$ and $\alpha \in \End A_B$, define the element in $\sum_i x_i \otimes_B \mathcal{F}(f_i) \in A \otimes_B A$ via an obvious identification of $\Hom (A_B,B_B) \into \End A_B$.  
Note that $$\sum_i x_i \mathcal{F}(f_i) = \sum_i \mathcal{F}(\sum_i \lambda_{x_i} f_i) = \mathcal{F}(\id_A ) = 1_A.$$
Since $\sum_i x_i \otimes_B f_i \in A \otimes_B \Hom (A_B,B_B) \cong \End A_B$ is central with respect to the $A$-bimodule actions, and $\mathcal{F}$ is an $A$-$A$-bimodule homomorphism, we compute for all $a \in A$, 
$$\sum_i ax_i \otimes_B \mathcal{F}(f_i) = \sum_i x_i \otimes_B \mathcal{F}(f_i \lambda_a) = 
\sum_i x_i \otimes_B \mathcal{F}(f_i) a.$$
This then defines a mapping from conditional expectations to separability idempotents.

The two mappings above are inverses to one another, since
$$ \sum_i x_i \otimes_B f_i(e^1)e^2 = \sum_i x_i f_i(e^1) \otimes_B e^2 = e $$
for every separability idempotent $e \in (A \otimes_B A)^A$, and
$$ \alpha \longmapsto  \sum_i \alpha(x_i) \mathcal{F}(f_i) = \mathcal{F}(\sum_i \alpha(x_i)f_i) = \mathcal{F}(\alpha)$$
for every $\alpha \in \End A_B$.  It follows that either mapping is a one-to-one correspondence of separability idempotents with conditional expectations, as stated in the theorem.
\end{proof}

\begin{cor}
A right projective ring extension $A \| B$ is uniquely separable if and only if $\End A_B$ is a uniquely split extension of $A$.  
\end{cor}

\begin{example}
\begin{rm}
Let $G$ be a finite group and $H$ a subgroup of $G$, $k$ a commutative ring, $A$ the group algebra $kG$ and $B$ the subgroup algebra $kH$ in $A$.  Let $R$ denote the centralizer $A^B$.  Let $E: A \rightarrow B$ be the canonical bimodule projection
${}_BA_B \rightarrow {}_BB_B$ given by $E(\sum_{g \in G} a_g g) = \sum_{h \in H} a_h h$.  
Since $E$ is a Frobenius homomorphism with dual bases a transversal of $H$ in $G$, and their inverses, it follows that $E$ is a free $A$-generator of the $B$-dual of $A$, hence $E$ is a free $R$-generator of
$\Hom ({}_BA_B, {}_BB_B)$.  

Suppose that the centralizer $C_G(H)$ has an element $z$ not in $H$.  Then  with $r = 1_G + z$, the bimodule mapping $Er$ satisfies
$Er(1_A) = E(1 + z) = 1_A$, i.e., $Er$ is another bimodule projection of $A$ onto $B$. Hence, under this circumstance,  $A$ is not 
a unique split extension of $B$.
\end{rm}
\end{example}
The next theorem covers when group algebra extensions are unique split extensions.
\begin{theorem}
\label{theorem-uniquesplit}
Suppose $A \| B$ is a Frobenius extension of algebras over a field, where a Frobenius homomorphism
$E: A \rightarrow B$ satisfies $E(1) = 1$.  Then $A$ is a uniquely split extension of $B$ if and only if its centralizer equals the center,  $A^B = Z(B)$.
\end{theorem}
\begin{proof}
We prove the contrapositive.  Since $E$ freely generates $\Hom (A_B, B_B)$ over $A$, $E$ also freely generates
$\Hom ({}_BA_B, {}_BB_B))$ over $A^B := R$.  Of course, $R \supseteq Z(B)$.  

($\Leftarrow$) Suppose there are two distinct bimodule projections $E$ and $Er$ of $A$ onto $B$, for some $r \in R$.
Then $E(r) = 1$, so that $E(1-r) = 0$.  But $E|_R:  R \rightarrow Z(B)$ is a projection, so that
$$\dim Z(B) + \dim \ker E|_R = \dim R.$$
Since $0 \neq 1-r \in \ker E|_R$, it follows that $R \neq Z(B)$.

($\Rightarrow$) If $R \neq Z(B)$, we have $\dim R > \dim Z(B)$, so $\dim \ker E|_R > 0$.  Then there is $r \in R - Z(B)$ such that $E(r) = 0$.  Then $1+ r \neq 1$, so that $E \cdot (1+r) \neq E$, but $E(1+r) = 1$,
whence $E \cdot (1+r)$ is another bimodule projection $A \rightarrow B$, i.e., $A$ is a nonuniquely split extension of $B$.  
\end{proof}

\section{Uniquely separable Frobenius extensions}
\label{three}

Recall that a Frobenius (ring) extension $A \supseteq B$ is characterized by having a (Frobenius)
homomorphism $E: A \rightarrow B$ in $\Hom ({}_BA_B, {}_BB_B)$ with elements (dual bases)
$x_i, y_i \in A$ ($i = 1,\ldots,n$) such that $\id_A = \sum_{i=1}^n E(-x_i)y_i = \sum_{i=1}^m x_iE(y_i -)$.  Equivalently, $A_B$ is finite projective and $A \cong \Hom (A_B, B_B)$ as natural $B$-$A$-bimodules:  see \cite{NEFE} for more details.  For example, given a group $G$ and subgroup $H$
of finite index $n$ with right coset representatives $g_1,\ldots,g_n$, $K$ an arbitrary commutative
ring, the group algebra $A = KG$ is a Frobenius extension of the group subalgebra $B = KH$
with $E: A \rightarrow B$ the obvious projection defined by $E(\sum_{g \in G} a_g g) = \sum_{h \in H} a_h h$ and dual bases $x_i = g_i^{-1}, y_i = g_i$.  

The element $\sum_i x_iy_i := [A:B]_E$ is for all Frobenius extensions in the center $Z(A)$ by the short computation, $\sum_i x_iy_ia = \sum_{i,j} x_iE(y_iax_j)y_j = \sum_j ax_j y_j$.  This element
is sometimes called the $E$-index, independent of the choice of dual bases for the Frobenius homomorphism \cite{LK95, NEFE}.
In the group algebra extension example above note that $\sum_{i=1}^n x_i y_i = n1$, which is invertible if and only if $A \supseteq B$ is a separable extension \cite{SH}.  In this case a separability idempotent 
is given by $\frac{1}{|G:H|} \sum_{i=1}^n g_i^{-1} \otimes_{KH} g_i$.  We study in general terms
when this separability idempotent is unique (e.g., in contrast to the $n$ different separability idempotents in Example~\ref{example-matrixalgebra}).  Call a separable Frobenius extension satisfying the hypotheses in the theorem a \textit{uniquely separable Frobenius extension}.

\begin{theorem}
\label{th-SH}
Given Frobenius extension $A \supseteq B$ with Frobenius homomorphism $E: A \rightarrow B$
and dual bases $x_i, y_i \in A$ such that $\sum_i x_i y_i := [A: B]_E$ is an invertible element in $Z(A)$, then
$Z(A) = A^B$ if and only if $A \supseteq B$ has the unique separability element ${[A:B]_E}^{-1} \sum_i x_i \otimes_B y_i$.  
\end{theorem}
\begin{proof}
Let $R$ denote the centralizer $A^B$.  
Recall that $A \otimes_B A \cong \End A_B$ via $a \otimes_B a' \mapsto \lambda_a \circ E \circ \lambda_{a'}$ (with inverse $f \mapsto \sum_i f(x_i) \otimes_B y_i$).  Consequently,
$(A \otimes_B A)^A \cong \End {}_AA_B \cong R$ via $e = e^1 \otimes_B e^2 \mapsto
e^1E(e^2-) \mapsto e^1E(e^2)$.  These have inverse mappings given by
$r \mapsto \rho_r \mapsto \sum_i x_i r \otimes_B y_i$.  

($\Rightarrow$) Suppose $e_j = \sum_i x_i r_j \otimes_B y_i$ for $r_j \in R, j = 1,2$ are two
separability idempotents in $(A \otimes_B A)^A$.  In particular, $\sum_i x_i r_j y_i = 1$ for $j = 1,2$.
But we assume $R = Z(A)$, so $r_j [A:B]_E = 1$ for both $j = 1,2$.  Thus, $r_1 = [A:B]_E^{-1} = r_2$.  

($\Leftarrow$) Denote by $e = e^1 \otimes_B e^2$ the unique separability idempotent.  Then
$\pi: R \rightarrow Z(A)$ defined by $\pi(r) = e^1 r e^2$ is a $Z(A)$-linear projection onto $Z(A)$.
If $\ker \pi \neq 0$, then there is $0 \neq x \in R$ such that $\pi(x) = 0$.  Then $r = [A:B]_E^{-1} + x \in R$
and $\sum_i x_i r \otimes_B y_i = e + f$, where $0 \neq f \in (A \otimes_B A)^A$
(since $E(x_i x)y_i = \eta(x) \neq 0$ where $\eta: R \rightarrow R$ is the Nakayama automorphism) and $\mu(f) =
[A:B]_E \pi(x) = 0$.  It follows that $e + f$ is  another separability element, contradicting the hypothesis.
Hence $\ker \pi = 0$ and $R = Z(A)$.
\end{proof}

In connection with the proof of the next proposition, we recall the next theorem.   
\begin{theorem}\cite[Burciu]{B}
\label{theorem-burciu}
If $B \subseteq A$ is a subalgebra pair of \textit{semisimple complex} algebras, then $d(B,A) = 1$
if and only if the centers satisfy $$Z(B) \subseteq Z(A).$$
\end{theorem}
\begin{proof}
By Lemma~\ref{lemma-d1}, the implication 
$\Rightarrow$ is more generally true for any centrally projective ring extension.
 It is known that a semisimple complex algebra-subalgebra pair form a split, separable Frobenius algebra extension \cite{BuK, BKK, NEFE}.  The implication $\Leftarrow$ is established in \cite{B}. 
\end{proof}
\begin{cor}
\label{cor-contain}
If $B \subseteq A$ is a uniquely separable Frobenius extension of semisimple complex algebras, then
$d(B,A) = 1$.  
\end{cor}
\begin{proof}
The centralizer $R$ of an extension always contains the centers $Z(A)$ and $Z(B)$.  From the hypothesis and the theorem, $Z(A) = R \supseteq Z(B)$, and apply the theorem.  
\end{proof}
From \cite{SH, S} and group theory, the following is a common setup.
\begin{prop}
If the dual bases $x_i, y_i$ of a separable Frobenius extension $A \supseteq B$ (where $\sum_i x_iy_i$
is invertible) may be chosen from $Z(A)$, then $A \supseteq B$ is uniquely separable Frobenius and has depth $1$.  
\end{prop}
\begin{proof}
Since any (Casimir) element in $(A \otimes_B A)^A$ may be written as $\sum_i x_i r \otimes_B y_i$
where $x_i, y_i \in Z(A)$,  if we moreover assume that $\sum_i x_i r y_i = 1$, then $r(\sum_i x_i y_i) = 1$, whence $r$ is unique.  This proves that $A \supseteq B$ is a uniquely separable Frobenius extension. 

From $a = \sum_i E(ax_i)y_i$ for all $a \in A$,  where $y_i \in R$ and $E(-x_i) \in \Hom ({}_BA_B, {}_BB_B)$ for each $i$.  This characterizes central projectivity of $A$ over $B$, i.e., $d(B,A) = 1$.  
\end{proof}

We briefly dispose of a question about what are uniquely split group algebra extensions.

\begin{theorem}
If $B \subseteq A$ is a subalgebra pair of semisimple complex algebras, then $d_h(B,A) = 1$ if and only if
the centers satisfy $$Z(A) \subseteq Z(B).$$
\end{theorem}
\begin{proof}
Recall that an algebra extension of semisimple complex algebras is a split separable Frobenius extension. 
Thus the natural module $A_B$ is a progenerator, whence $B$ and $\End A_B$ are Morita equivalent. 
The implication $\Rightarrow$ is more generally true by Lemma~\ref{lemma-hsep}.  Assume now that $Z(A) \subseteq Z(B)$. 
The endomorphism ring extension $A \into \End A_B$ has depth 1 by Theorem~\ref{theorem-burciu}, since $Z(A) \subseteq Z(B) \cong Z(\End A_B)$ by Morita theory and graphically.  Then $B \subseteq A$ has H-depth 1 by  \cite[Theorem 4.2]{LK2011}. 
\end{proof} 
\begin{cor}
A uniquely split group complex algebra extension $A \supseteq B$ of finite groups is a trivial extension, $A = B$.
\end{cor}
\begin{proof}
From Theorem~\ref{theorem-uniquesplit}, $Z(B) = A^B$.  Then $Z(A) \subseteq Z(B)$.  It follows from the last theorem that
$A$ is H-separable over $B$.  By \cite[Corollary 3.3]{K2013} $A = B$.  
\end{proof}
The corollary is more generally true for algebra extensions of semisimple  Hopf algebras over an algebraically closed field of characteristic zero, as the reader may check in the cited sources.  
\section{Group algebra extensions}
\label{four}

Let $A = kG$ where $G$ is a finite group with subgroup $H < G$, and let $B = kH \subseteq A$ where $k$ is a field containing the inverse of $|G: H| 1$.  Then $A \supseteq B$ is a split, separable Frobenius extension.  If
$g_1,\ldots,g_n \in G$ is a right transversal of $H$ in $G$, then $e = \frac{1}{|G:H|} \sum_{i=1}^n {g_i}^{-1} \otimes_B g_i$ is a separability element.  The next theorem is a consequence of a theorem
by Singh-Hanna in \cite{SH}, which we show is also a consequence of Theorem~\ref{th-SH}.  
\begin{theorem}[\cite{SH}]
A separable finite group algebra extension $A \supseteq B$ has unique separability element $e \in
A \otimes_B A$ if and only if it satisfies the property: \newline
(S) Any conjugacy class in $G$ is an $H$-orbit.
\end{theorem}
\begin{proof}
It is well-known and easy to check that the sum of elements in a conjugacy class is in the center of a group algebra $kG$, and that the dimension of the center is equal to the number of conjugacy classes of $G$. Similarly,
the sum of elements in an $H$-orbit of $G$ (under conjugation) is in the centralizer $R = A^B$ of $B$ in $A$, and the dimension of $R$ is the number of distinct $H$-orbits in $G$.  Thus the condition is equivalent to
$R = Z(A)$, which in turn by Theorem~\ref{th-SH} is equivalent to uniqueness of the separability element $e$. 
\end{proof}

For example, $G = Q_8 = \{ \pm 1, \pm i, \pm j, \pm k \} \geq H = \{ \pm 1, \pm i \}$ does not satisfy property (S).
In \cite[Theorem 3.8]{SH} it is shown that for groups $G$ of order less than $64$, but different from  $48$,  each  subgroup $H$  satisfying the  property (S) also satisfies $G = HZ(G)$; see also \cite[Singh]{SS}.  
Notice that the following holds for centralizers of elements in a finite group $G$ with subgroup $H$.

\begin{lemma}[\cite{SH}]
\label{lemma-centralizers}
$H$ satisfies (S) in $G$ if and only if $G = HC_G(a)$ for every $a \in G$.
\end{lemma}
\begin{proof}
If for every $g,a\in G$, there is $h \in H$ such that $gag^{-1} = hah^{-1}$,
then $h^{-1}g a = a h^{-1} g$, so that $h^{-1} g \in C_G(a)$.  Then $g \in HC_G(a)$,
whence $G = HC_G(a)$ for every $a \in G$.  

Conversely, given arbitrary $g, a \in G= HC_G(a)$, there is $h \in H, x \in C_G(a)$ such that
$g = hx$.  Then $gag^{-1} = hxax^{-1}h^{-1} = hah^{-1}$.  Thus $H < G$ satisfies (S).    
\end{proof}

Compare this to the characterization of depth $1$ in \cite{BK} of group algebra extensions over a field $k$ having characteristic zero:  $d_k(H,G) = 1$ if and only if $G = HC_G(h)$ for every $h \in H$, a weaker condition than the one in the lemma.  Thus for group algebras we may improve on Corollary~\ref{cor-contain} by noting the following.  
\begin{prop}
Suppose $k$ is a field of characteristic zero and $H$ a subgroup of a finite group $G$.  If
$kG \supseteq kH$ is uniquely separable, then it has depth $1$.
\end{prop}
 \begin{prop}
If a subgroup $H$ of a finite group $G$ satisfies $HZ(G) = G$, then $KG \supseteq KH$ is uniquely separable over any commutative ring $K$ where $|G: H| 1$ is invertible, and $KG \supseteq KH$ has depth $1$. 
\end{prop}
\begin{proof}
Given $g,x \in G$, there is $h \in H$ and $z \in Z(G)$, such that $g = hz$.  Then $gxg^{-1} = hxh^{-1}$,
so $H$ satisfies the property (S) in $G$.  Also $Z(G) \subseteq C_G(H)$, so that $G = HZ(G) \subseteq HC_G(H)$.  Then $G = HC_G(H)$, the sufficient condition in \cite[1.12]{BK} for $\Z H \subseteq \Z G$
to have depth $1$, and therefore $d_K(H,G) = 1$ for any commutative ring $K$.  
\end{proof}


\end{document}